\documentclass[12pt, reqno]{amsart}
\usepackage{amscd}
\usepackage{latexsym}
\usepackage{amssymb}
\usepackage{amsmath}
\usepackage{amsthm}
\usepackage{indentfirst}
\usepackage{pifont}
\usepackage{graphicx}
\usepackage{geometry}
\usepackage{subfig}
\usepackage{cases}
\usepackage{hyperref}
\usepackage{cite}
\usepackage{color}

\numberwithin{equation}{section}

\newtheorem{theorem}{Theorem}[section]
\newtheorem{lemma}{Lemma}[section]

\theoremstyle{definition}

\newtheorem{proposition}{Proposition}[section]

\theoremstyle{remark}

\newcommand{\di}{\displaystyle}
\newcommand{\x}{\xi}
\newcommand{\RR}{\mathbb{R}}
\newcommand{\mb}{\mathbf}

\begin{document}

\title[$L^2$-contraction of large shock]{\bf $L^2$-contraction and asymptotic stability of large shock for scalar viscous conservation laws}

\author[Vasseur]{Alexis F. Vasseur}
\address[Alexis F. Vasseur]{\newline Department of Mathematics, \newline The University of Texas at Austin, Austin, TX 78712, USA}
\email{vasseur@math.utexas.edu}

\author[Wang]{Yi Wang}
\address[Yi Wang]{\newline State Key Laboratory of Mathematical Sciences and Institute of Applied Mathematics, AMSS, Chinese Academy of Sciences, Beijing 100190, P. R. China
\newline
and School of Mathematical Sciences, University of Chinese Academy of Sciences,
\newline Beijing 100049, P. R. China}
\email{wangyi@amss.ac.cn}

\author[Zhang]{Jian Zhang}
\address[Jian Zhang]{\newline School of Mathematical Sciences,
\newline
Chengdu University and Technology, Chengdu 61005, P. R. China}
\email{zhangjian2020@amss.ac.cn}

\date{}
\maketitle

\noindent{\bf Abstract}:  We investigate $L^2$-contraction and  time-asymptotic stability of large shock for scalar viscous conservation laws with polynomial flux.  For the strictly convex flux $f(u)=u^p $ with $2\leq p \leq 4$, we can prove $L^2$-contraction and  time-asymptotic stability of arbitrarily large viscous shock profile in $H^1$-framework by using $a$-contraction method with time-dependent shift and suitable weight function. Additionally, if the initial perturbation belongs to $L^1$, then $L^2$ time-asymptotic decay rate $t^{-\frac{1}{4}}$ can be obtained. 

 \section{Introduction}

 We are concerned with  $L^2$-contraction and time-asymptotic stability of arbitrarily large shock for the following scalar viscous conservation laws with polynomial flux:
 \begin{equation}\label{1.1}
\left\{\begin{array}{l}
u_t+f(u)_x= u_{xx},\quad  f(u)=u^p, \ \ (t,x)\in  \mathbb{R}_{+}\times \mathbb{R},\\ [1mm]
u(0,x)=u_0(x),\\[1mm]
\displaystyle \lim_{x\rightarrow \pm \infty}u_0(x)=u_\pm.
\end{array}\right.
\end{equation}
where $u_0(x)$ is the given initial data and $u_{\pm}\in\mathbb{R}$ are the prescribed far-field states. We focus on the case that the asymptotic state of the solution to \eqref{1.1} is the viscous shock wave. Therefore, it is assumed that $p>1$ such that $f(u)=u^p$ is strictly convex for $u>0$, and that
\begin{equation}\label{sc}
0<u_+<u_-.
\end{equation}

Remark that for the special case $p=2,$ that is, the classical Burgers equation, $f^{\prime\prime}(u)=2$ and $f(u)=u^2$ is always strictly convex for any $u\in\RR$, and then we only need to assume that $u_+<u_-$.

It can be expected that the large-time asymptotic behavior of the solution to \eqref{1.1} and \eqref{sc}  is determined by the following viscous shock profile $U(x-st)$
  \begin{equation}\label{1.2}
\left\{\begin{array}{l}
-sU^\prime+f(U)^\prime=U^{\prime\prime},\quad ^\prime=\frac{d}{d\x}, \quad \x=x-st,\\[2mm]
\displaystyle U(\pm\infty)=u_\pm,
\end{array}\right.
\end{equation}
 where $s$ is the shock speed determined by the Rankine-Hugoniot condition:
  \begin{equation}\label{1.3}
s=\frac{f(u_+)-f(u_-)}{u_+-u_-}.
\end{equation}

Integrating \eqref{1.2} over $(\pm\infty, \xi]$, we can get the following first order ODE
  \begin{equation}\label{1.4}
  \begin{aligned}
U^\prime&=h(U):=-s(U-u_\pm)+f(U)-f(u_\pm)\\
&=(U-u_\pm)\left[\frac{f(U)-f(u_\pm)}{U-u_\pm}-\frac{f(u_+)-f(u_-)}{u_+-u_-}\right]<0.
\end{aligned}
\end{equation}
Note that the strict convexity of  the flux $f(u)$ implies the above decreasing monotonicity of the viscous shock profile $U(\xi)$. Moreover, the existence of the viscous shock profile $U(\xi)$ to \eqref{1.2} is standard and it is unique up to any constant translation.

The stability of viscous shock wave for conservation laws has been extensively studied since the pioneer works of Hopf \cite{Hopf} and Il'in-Oleinik \cite{IO} for one-dimensional (1D) scalar equation. In 1976, Sattinger \cite{S} introduced a semigroup approach to establish the stability of viscous shock waves to 1D parabolic equations, including \eqref{1.1}, in certain weighted spaces. Then Matsumura-Nishihara \cite{MN} and Goodman \cite{Go} independently proved the time-asymptotic stability of viscous shock profile for 1D isentropic Navier-Stokes equations and viscous conservation laws with artificial viscosity respectively, under the zero mass conditions such that the anti-derivative method can be applied. Meanwhile, Nishihara \cite{N} proved the point-wise stability of viscous shock to 1D Burgers equation by virtue of the Hopf-Cole transformation and Kawashima-Matsumura \cite{KAM} further obtained the convergence rate for the viscous shock to scalar equation by weighted energy method and the stability of viscous shock to both 1D full Navier-Stokes-Fourier equations and the discrete Broadwell model system. Note that in all the above mentioned time-asymptotic results for viscous shock to the system case, the crucial zero mass conditions are imposed to the initial perturbations such that the anti-derivative variables  can be well-defined. Then Liu \cite{Liu},  Szepessy-Xin \cite{SX} and Liu-Zeng \cite{LZ} removed the zero mass conditions in \cite{MN, Go, KAM} by introducing the constant shift on the viscous shock and the 
coupled diffusion waves in the transverse characteristic fields. And Mascia-Zumbrun \cite{MZ} established the spectral stability of viscous shocks for the 1D compressible Navier-Stokes system, relaxing the zero mass conditions to a slightly weaker spectral condition. More recently, Kang-Vasseur-Wang \cite{ W1,W2} proved the generic Riemann solutions (containing viscous shock wave, rarefaction wave, and even viscous contact wave) to both barotropic compressible Navier--Stokes equations \cite{W2} and full compressible Navier-Stokes(-Fourier) equations  \cite{W1} with the help of $a$-contraction method for the stability of viscous shock wave with time-dependent shift invented in \cite{KV}. Remark that both \cite{W1} and \cite{W2} solve the long-standing open problems for the time-asymptotic stability of the composite waves of Riemann profiles.
 In addition, Freistuhler-Serre proved the $L^1$-stability of viscous shock waves \cite{FS} (see also Serre \cite{DS}) and Kenig-Merle proved $L^p$-stability($1\leq p\leq +\infty$) of viscous shock waves \cite{KM1}  to scalar equation. For multi-dimensional case, Goodman \cite{G} proved nonlinear stability of planar shock profile for viscous scalar conservation laws by using the anti-derivative techniques and the shift function depending on both the time and the transverse spatial variables.  Then Humpherys–Lyng–Zumbrun \cite{HLZ} proved the spectral
stability of the planar viscous Navier-Stokes shock by the numerical Evans-function method in $\mathbb{R}^3$  and one can refer to the survey paper by Zumbrun \cite{Z} for the related results and the references therein. Very recently, Wang-Wang \cite{WW} proved the time-asymptotic stability of planar weak shock profile to 3D barotropic compressible Navier-Stokes equations in $\RR\times \mathbb{T}^2$ under general $H^2$-initial perturbations by using $a$-contraction method.

On the other hand, for scalar conservation laws with or without viscosity, Kružkov \cite{K} proved the famous $L^1$-contraction stability in quite general case. However, Kružkov's theory is not valid for $L^p$-contraction $(\forall p>1)$, in particular, in the physical $L^2$-norm. 
With the time-dependent shift $\mb{X}(t)$ and possible weight function $a$, Kang-Vasseur prove $L^2$-contraction of arbitrarily large shock to both 1D general inviscid system of hyperbolic conservation laws \cite{KV2016}  and 1D viscous Burgers equation \cite{KV}, and the latter result is extended by Kang \cite{K} to scalar viscous conservation laws with general strictly convex flux, provided that the viscous shock wave strength was sufficiently small. Then Kang-Vasseur-Wang \cite{KVW2} prove $L^2$-contraction of large planar shock to scalar multi-dimensional viscous conservation laws by a special transformation and Kang-Vasseur \cite{VKM} obtain $L^2$-contraction of small shock to 1D compressible barotropic Navier-Stokes equations. Kang-Oh \cite{KO} obtained $L^2$ decay for large perturbations of weak viscous shock for multi-dimensional Burgers equation. Note that $a$-contraction theory has been well established at the inviscid level uniformly in the shock amplitude \cite{KV2016}, and at the viscous level only for small shocks \cite{KV} even with small perturbations, except Kang-Vasseur's work \cite{KV} on specific Burgers’ equation for arbitrarily large shock. In fact, in the context of the viscous model, $a$-contraction results in the large shock setting (even in the scalar case) remain largely open.
Very recently, Blochas-Cheng \cite{BC} showed that the $a$-contraction property fails for large shock to certain kinds of viscous conservation laws, which showcases a “viscous
destabilization” effect in the sense that the $a$-contraction property is verified for the inviscid model for arbitrarily shock, but can fail for the viscous one, and raised the open question of whether this $a$-contraction property of large shock holds for scalar viscous conservation laws with polynomial fluxes. 

In the present paper, we aim to address the question in \cite{BC} under small $H^1$ framework around the shock profile and further prove the time-asymptotic stability of this large shock with the time decay rate. Based on $a$-contraction method with suitably chosen time-dependent shift and the weight function $a$, crucially depending on the shock wave strength, we can establish $L^2$-contraction for any large shock to \eqref{1.1} with strictly convex fluxes of the polynomial form $f(u)=u^p$ ($2\leq p \leq 4$) and suitably small $H^1$ perturbations, and further prove $L^\infty$ time-asymptotic stability of this large shock profile with the time decay rate if the initial $H^1$-perturbation additionally lies in $L^1$. In fact, our proof is motivated by the recent work of Huang-Wang-Zhang \cite{HWZ} for the time-asymptotic stability of composite wave of large viscous Oleinik shock and rarefaction wave for the cubic non-convex scalar viscous conservation laws.


Precisely, our main result can be stated as follows.
  
\begin{theorem}\label{T1.1}  For any given $u_\pm$ satisfying \eqref{sc}, let $U(\xi)$ be viscous shock wave defined in \eqref{1.2}. Then there exists a positive constant  $\epsilon^*$ such that if  the initial data $u_0$ satisfies 
 \begin{equation}\label{1.5}
\Big{ \|}u_0(\cdot)-U(\cdot-x_0)\Big{\|}_{H^1(\mathbb{R})}<\epsilon^*,
  \end{equation}
for any initial shock location $x_0\in\mathbb{R}$, then Cauchy problem \eqref{1.1} with $p\in [2,4]$ admits a unique global-in-time classical solution $u$. Moreover, there exists an absolutely continuous shift $\mathbf{X}(t)$ (defined in \eqref{2.5}) and a smooth weight function $a(U(\x)): \RR\rightarrow \RR^{+}$ (defined in \eqref{2.4})  such that the following $L^2$-contraction of arbitrarily large viscous shock holds
 \begin{equation}\label{1.6}
\frac{d}{dt} \int_{\RR}  a\big(U(x-st-\mathbf{X}(t))\big)|u(t, x)-U(x-st-\mathbf{X}(t))|^2dx\leq 0, \quad \forall t>0.
\end{equation}
Moreover, the time-asymptotic stability of large viscous shock holds
 \begin{equation}\label{1.7}
 \displaystyle\lim_{t\rightarrow+\infty} \displaystyle \sup_{x\in\mathbb R}\Big|u(t,x)- U(x-st - \mathbf{X}(t)\Big|= 0,
 \end{equation}
with
 \begin{equation}\label{1.8}
 \displaystyle\lim_{t\rightarrow+\infty}|\dot{\mathbf{X}}(t)|= 0.
 \end{equation}
 In addition, if $u_0(x)-U(x-x_0)\in L^1(\RR)$, we have the following $L^2$-time decay rate
  \begin{equation}\label{1.9}
   \|u(t, \cdot)-U(\cdot-st-\mathbf{X}(t))\|_{L^2(\RR)}\leq \frac{C_*\|u_0(\cdot)-U(\cdot-x_0)\|_{L^2(\RR)}}{\di 1+C_*t^{\frac{1}{4}}\|(\cdot)-U(\cdot-x_0)\|_{L^2(\RR)}},
 \end{equation}
 where $C_*:=C\big[1+\|u_0(\cdot)-U(\cdot-x_0)\|_{L^1(\RR)}+\|u_0(\cdot)-U(\cdot-x_0)\|_{H^1(\RR)}\big]$. 
 
 \
 
 {\bf Remark 1.1.} Since the weight function $a(U(\x))$ is bounded from below and above, $L^2$-contraction \eqref{1.6} implies the uniform $L^2$-stability 
  \begin{equation}\label{1.10}
 \int_{\RR} |u(t, x)-U(x-st-\mathbf{X}(t))|^2dx\leq C^*  \int_{\RR} |u_0(x)-U(x-x_0)|^2dx,\quad \forall t>0, 
\end{equation}
where positive constant $C^*=C^*(u_+,u_-,p )$ is independent of the time $t>0$.
\end{theorem} 
  {\bf Remark 1.2.}  The shift ${\mathbf{X}}(t)$ is proved to satisfy the time-asymptotic behavior \eqref{1.8}, which implies 
   \begin{equation}\label{1.11}
 \displaystyle\lim_{t\rightarrow+\infty}\frac{\mathbf{X}(t)}{t}= 0, \notag
 \end{equation}
 that is, the shift $\mathbf{X}(t)$ grows at most sub-linearly with respect to the time $t$. Therefore, the shifted viscous shock wave $U\big(x-st - \mathbf{X}(t))$ keeps the original traveling wave profile time-asymptotically.
 
 
 \

 \noindent{\bf Notations.} Throughout this paper,  several positive  generic  uniform-in-time constants are denoted by $C$. 
  Denote   $L^p(\mathbb{R})(1\leq p\leq+\infty) $ and $H^1(\mathbb{R})$ as the usual Lebesgue space and Sobolev space in $\mathbb{R}$ with the norm
$$
\|f\|_{L^p(\mathbb{R})}:=\Big(\int_{\mathbb{R}} |f|^pd\xi\Big)^\frac{1}{p},\  1\leq p<+\infty,\quad  \|f\|_{L^\infty(\mathbb{R})}:={\rm ess} \sup_{x\in \mathbb{R}}|f(x)|.
$$
and
$$
\|f\|:=\|f\|_{L^2(\mathbb{R})},\ \ \ \|f\|_{H^1(\mathbb{R})}:  =(\|f\|^2+\|f_\xi\|^2)^\frac{1}{2}. \notag
$$
   
   For any function $f:\RR^+\times\RR\rightarrow \RR$ and the shift function $\mathbf{X}(t)$, we denote
    \begin{equation}
f^{\pm\mathbf{X}}(t,\xi):=f(t,\xi\pm\mathbf{X}(t)). \notag
   \end{equation}
 \section{Preliminaries and main result} 
   In this section, we start with the construction of weight function $a(U(\x))$ and shift function $\mathbf{X}(t)$. Then, we present the local existence of the solution and the uniform-in-time a priori estimates, whose proofs are given in the subsequent sections. Finally, we give the  proof of Theorem \ref{T1.1} by the continuity arguments.

For convenience, we rewrite the equation \eqref{1.1} through the coordinates transformation $(t,x)\rightarrow(t,\xi=x-st)$, and then $u(t,\xi):=u(t,\xi+st)=u(t, x)$ satisfies
\begin{equation}\label{2.1}
u_t-s u_\x+f(u)_\x= u_{\x\x}.
\end{equation}
And  $U^{-\mb X}:=U(\x-\mathbf{X}(t))$ satisfies 
\begin{equation}\label{2.2}
U^{-\mb X}_t+ \mathbf{\dot X}(t)U^{-\mb X}_\x-sU^{-\mb X}_\x+f(U^{-\mb X})_\x=U^{-\mb X}_{\x\x}.
\end{equation}
Then the perturbation $\phi(t,\x):=u(t,\x)-U(\x-\mathbf{X}(t))$ satisfies
\begin{equation}\label{2.3}
\left\{\begin{array}{l}
\phi_t-s\phi_\xi-\mathbf{\dot X}(t)U^{-\mb X}_\x+[f(\phi+U^{-\mb{X}})-f(U^{-\mb{X}})]_\xi=\phi_{\x\x},\\[3mm]
\phi(0,\xi)=\phi_0(\xi):=u_0(\xi)-U(\xi-x_0),
\end{array}\right.
\end{equation}
for any initial shock location $x_0\in \mathbb{R}$.
In the sequel, we always assume that $p\in [2,4]$ for the polynomial flux $f(u)=u^p$.

\subsection{Construction of weight function}
   
Define the weight function $ a(U(\x))$ as
  \begin{equation}\label{2.4}
a(U(\x)):=\frac{(u_--u_+)h(U)}{(U-u_+)(U-u_-)}=\frac{f(U)-f(u_-)}{U-u_-}-\frac{f(U)-f(u_+)}{U-u_+},
 \end{equation}
It is easy to check that $a\in C^\infty(\mathbb{R})$, $a(U(\x))>c>0$ for any $\x \in\mathbb{R}$, that is, $U  \in (u_+,u_-)$, and $\|a(U(\x))\|_{C^1(\RR)}\leq C$. Notice that for the Burgers equation case $(p=2)$, the above weight function $a=u_--u_+$ is exactly the shock wave strength, which means that there is no need for the weight and is consistent with the classical result \cite{KV} for Burgers equation.

\subsection{Construction of shift function} Define the shift function $\mathbf{X}(t)$ as the solution to the following ODE:
\begin{equation}\label{2.5}
\left\{\begin{array}{l}
\di \dot{\mathbf{X}}(t)=-\frac{4}{(u_--u_+)^2}\int_{\mathbb{R}}a(U^{-\mb{X}}(\x)) U^{-\mb{X}}_\xi(\x)  \phi(t,\x) d\xi,\\[4mm]
\mathbf{X}(0)=x_0,
\end{array}\right. 
\end{equation}  
where the initial shock location $x_0$ can be chosen arbitrarily with or without zero mass condition. 


 Denote
 $$
 \phi^{\mb{X}}(t,\x) := \phi(t,\x+\mb{X}(t)) 
 =u(t,\x+\mb{X}(t))-U(\x).$$
Then the perturbation equation \eqref{2.3} can be rewritten as
\begin{equation}\label{pe}
 \phi^{\mb X}_t-\dot{\mb X}(t)(\phi^{\mb X}_\xi+U_\x)-s \phi^{\mb X}_{\x}+\big[f(\phi^{\mb X}+ U)-f(U)\big]_\x  =\phi^{\mb X}_{\x\x}.
\end{equation}
By \eqref{2.5}, we have
\begin{equation}\label{2.5+}
\di \dot{\mathbf{X}}(t)=-\frac{4}{(u_--u_+)^2}\int_{\mathbb{R}} a(U(\x)) U_\xi(\x) \phi^{\mb{X}}(t,\x)  d\xi.
\end{equation} 
Substituting \eqref{2.5+} into \eqref{pe} gives that
\begin{equation}\label{pe+}
 \phi^{\mb X}_t+\frac{4(\phi^{\mb X}_\xi+U_\x)}{(u_--u_+)^2}\int_{\mathbb{R}} a(U) U_\xi  \phi^{\mb{X}} d\xi-s \phi^{\mb X}_{\x}+\big[f(\phi^{\mb X}+ U)-f(U)\big]_\x  =\phi^{\mb X}_{\x\x},
\end{equation}
with the initial condition
\begin{equation}\label{in}
 \phi^{\mb X}(t=0,\xi)=u_0(\x+x_0)-U(\x):= \phi^{\mb X}_0(\xi).
 \end{equation}
Note that both the non-local equation \eqref{pe+} for  $\phi^{\mb X}$ and the initial value $\phi^{\mb X}_0(\xi)$ in \eqref{in} are independent of the definition of the shift function $\mb X(t),$ even though we still write as  $\phi^{\mb X}$.

 We reformulate the Cauchy problem at general initial time $\tau\geq 0$, that is,
 \begin{equation}\label{aaa}
\left\{\begin{array}{l}
\displaystyle\phi^{\mb X}_t+\frac{4(\phi^{\mb X}_\xi+U_\x)}{(u_--u_+)^2}\int_{\mathbb{R}} a(U) U_\xi  \phi^{\mb{X}} d\xi-s \phi^{\mb X}_{\x}+\big[f(\phi^{\mb X}+ U)-f(U)\big]_\x  =\phi^{\mb X}_{\x\x},\quad t>\tau,\\[5mm]
 \phi^{\mb X}(t=\tau,\xi):= \phi^{\mb X}_\tau(\xi).
\end{array}\right. 
\end{equation}
  The local existence of the strong solution $\phi^{\mb X}\in C([\tau, \tau+t_0];H^1(\RR))\cap L^2(\tau,\tau+t_0; H^2(\RR))$ to the scalar nonlocal equation \eqref{pe+} can be stated as follows,  whose proof will be given in Appendix.

\begin{proposition}\label{T3.1}(Local existence) For any $M>0$, there exists a positive constant $t_0=t_0(M)$ which is independent of $\tau$ such that if  $\|\phi^{\mb X}_\tau\|_{H^1(\mathbb{R})}\leq M$, then the Cauchy problem \eqref{pe+}-\eqref{in} has a unique strong solution $\phi^\mb{X}(t,\xi)$ on the time interval $[\tau,\tau+t_0]$ satisfying
\begin{equation}\label{3.17}
\left\{\begin{array}{l}
\phi^\mb{X}(t,\xi)\in C([\tau,\tau+t_0];H^1(\RR))\cap L^2(\tau,\tau+t_0; H^2(\RR)),\\[3mm]
 \displaystyle \sup_{t\in[\tau,\tau+t_0]}\| \phi^\mb{X}(t,\cdot)\|_{H^1(\mathbb{R})}\leq2M.
\end{array}\right.
\end{equation}
\end{proposition}
Then we can determine a unique shift function $\mb{X}(t)$ by ODE \eqref{2.5+} and the local solution $\phi(t,\x):=u(t,\x)-U(\x-\mb{X}(t))$.

\begin{proposition}\label{T3.2}(A priori estimates) For any given $u_\pm$ satisfying \eqref{sc}, there exists a positive constant $\epsilon_0>0$, such that if the Cauchy problem \eqref{2.3} has a solution $\phi\in C([0,T];H^1(\mathbb{R}))\cap L^2(0,T; H^2(\mathbb{R}))$ for some time $T>0$ with
\begin{equation}
\mathcal{N}(T):=\displaystyle \sup_{0\leq t\leq T}\|\phi\|_{H^1(\mathbb{R})}\leq \epsilon_0,  \notag
\end{equation}
then there exists a uniform-in-time positive constant $C_0$ such that for all $t\in[0,T]$,
\begin{equation}\label{3.18}
 \begin{aligned}
\|\phi(t)\|^2_{H^1(\mathbb{R})}+\int_0^t\|\phi_\xi\|^2_{H^1(\mathbb{R})}d\tau+\int_0^t\int_\mathbb{R} {\phi}^2|U^{-\mb{X}}_\xi|d \xi d\tau
\di+\int_0^t|  \mathbf{ \dot  X}(\tau)|^2d\tau\leq C_0\|\phi_0\|^2_{H^1(\mathbb{R})},
 \end{aligned}
\end{equation}
and
\begin{equation}\label{3.19}
| \mathbf{ \dot  X}(t)|\leq C_0\|\phi\|_{L^\infty(\mathbb{R})}.
\end{equation}
\end{proposition}
The proof of Proposition \ref{T3.2} will be given in next section. Based on Propositions \ref{T3.1} and \ref{T3.2}, we can prove Theorem \ref{T1.1} as follows.

\subsection{The continuity argument} By the continuity argument, we can extend the local solution to the global one for all $t\in[0,+\infty)$ as follows.  Define 
 \begin{equation}
 \epsilon^*:=\min \left\{\frac{\epsilon_0}{4}, \sqrt{\frac{\epsilon_0^2}{64C_0}}\right\},\qquad  M:=\frac{\epsilon_0}{4},\notag
\end{equation}
where $\epsilon_0$ and $C_0$ are given in Proposition \ref{T3.2}. By \eqref{1.5},  $\| \phi_0^{\mb X}\|_{H^1(\mathbb{R})}=\| \phi_0\|_{H^1(\mathbb{R})}< \epsilon^*\leq \frac{\epsilon_0}{4}(=M)$, and using the local existence result in Proposition \ref{T3.1}, there exists a positive constant $T_0=T_0(M)$ such that a unique solution on $[0,T_0]$ satisfing $\|\phi^{\mb X}(t,\cdot)\|_{H^1(\mathbb{R})}\leq  \frac{\epsilon_0}{2}$ for $t\in[0,T_0]$. 
Note that $ \|\phi(t,\cdot)\|_{H^1(\RR)}=\| \phi^{\mb X}(t,\cdot)\|_{H^1(\RR)}\leq \frac{\epsilon_0}{2}$ for $t\in[0,T_0]$.
Especially, since $ \mathbf X(t)$ is absolutely continuous and $\phi^{\mb X}\in C([0,T_0];H^1{(\mathbb{R})})$, we have $\phi\in C([0,T_0];H^1{(\mathbb{R})})$. Consider the maximal existence time:
 \begin{equation}
 T_M:=\sup\left\{t>0\Big|\displaystyle \sup_{\tau\in [0,t]}\|\phi(\tau, \cdot)\|_{H^1(\RR)}\leq \epsilon_0\right\}. \notag
\end{equation}
If $T_M<\infty$, then the continuity argument implies that  $\displaystyle \sup_{\tau\in [0,T_M]}\|\phi(\tau, \cdot)\|_{H^1(\mathbb{R})}=\epsilon_0$. By Proposition \ref{T3.2}, it holds that
 \begin{equation}
\sup_{\tau\in [0,T_M]}\|\phi(\tau, \cdot)\|_{H^1(\mathbb{R})}\leq \sqrt {C_0\|\phi_0\|^2_{H^1(\mathbb{R})}}\leq \frac{\epsilon_0}{8},  \notag
\end{equation}
which contradicts the fact $\displaystyle \sup_{\tau\in [0,T_M]}\|\phi(\tau, \cdot)\|_{H^1(\mathbb{R})}=\epsilon_0$.
Therefore, $$T_M=\infty.$$ By Proposition \ref{P2.1}, we have 

 \begin{equation}\label{1a+}
\frac{d}{dt} \int_{\RR} a\big(U(x-st-\mathbf{X}(t))\big)|u(t, x)-U(x-st-\mathbf{X}(t))|^2dx\leq 0, \forall t>0,
\end{equation}
which verifies $L^2$-contraction in Theorem \ref{T1.1}.
By using Proposition \ref{T3.2} again, we have 
\begin{equation}\label{3-2}
\begin{array}{ll}
\displaystyle \sup_{t>0}\|\phi(t,\cdot)\|^2_{H^1(\RR)}+\int_0^\infty\|\phi_\xi\|^2_{H^1(\RR)}d\tau+\int_0^\infty\int_\mathbb{R} {\phi}^2|U^{-\mb{X}}_\xi|d \xi d\tau\leq C_0\|\phi_0\|^2_{H^1(\RR)},
\end{array}
\end{equation}
and 
\begin{equation}\label{3-3}
| \mathbf{ \dot  X}(t)|\leq C\|\phi(t, \cdot)\|_{L^\infty(\mathbb{R})}, \ \ \ \forall t>0.
\end{equation}

 \subsection{Time-asymptotic behaviors} Now we justify the time-asymptotic behaviors \eqref{1.7} and \eqref{1.8}. First we set $$\di g(t):=\|\phi_\xi(t,\cdot)\|^2.$$From \eqref{3-2}, it is obvious  that $$g(t)\in L^1(0,+\infty).$$ Next we show that $$g^\prime(t)\in L^1(0,+\infty).$$ By \eqref{2.33}-\eqref{2.35}, we have
\begin{align}
\int_0^{+\infty}|g^\prime(t)|dt=&\int_0^{+\infty}\Big|\frac{d}{dt}\|\phi_\xi\|^2\Big|dt\notag\\
\leq &C\int_0^{+\infty}\Big(\|\phi_\xi\|_{H^1(\RR)}^2+|  \mathbf{ \dot  X}(t)|^2+\int_\mathbb{R} {\phi}^2|U^{-\mb{X}}_\xi |d \xi \Big)dt<+\infty.\notag
\end{align}
Hence 
\begin{equation}\label{3.20}
\displaystyle\lim_{t\rightarrow+\infty }g(t)=\displaystyle\lim_{t\rightarrow+\infty} \|\phi_\xi(t,\cdot)\|^2=0.
\end{equation}
By Sobolev inequality, we have 
\begin{equation}\label{3.21}
\displaystyle\lim_{t\rightarrow+\infty}   \|\phi(t, \cdot)\|_{L^{\infty}(\mathbb{R})}\leq\displaystyle\lim_{t\rightarrow+\infty} \sqrt 2\|\phi(t, \cdot)\|^\frac{1}{2}\|\phi_\xi(t, \cdot)\|^\frac{1}{2}=0.         
\end{equation}
By \eqref{3-3} and  \eqref{3.21}, it holds that
\begin{equation}\label{3.22}
 | \mathbf{ \dot  X}(t)|\leq C\|\phi(t, \cdot)\|_{L^\infty(\mathbb{R})}\rightarrow 0, \ \ \ {\rm as} \ \ \ t\rightarrow +\infty.
\end{equation}

 \subsection {Decay estimate} Besides \eqref{1.5}, if $u_0(x)-U(x-x_0)\in L^1(\RR)$, we first show that  $\mathbf{X}(t)-x_0$ is uniformly bounded with respect to $t\in \RR^+$.
Rewrite $u(t,x)-U(x-s t - \mathbf{X}(t))$ as follows:
\begin{equation}\label{I3}
  \begin{array}{ll}
\di u(t,x)-U(x-s t - \mathbf{X}(t))=&\di \underbrace{u(t,x)-U(x-s t-x_0) }_{I_1(t,x)}\\
&\di +\underbrace{U(x-s t-x_0)-U(x-s t - \mathbf{X}(t))}_{I_2(t,x)}.\\
\end{array}
\end{equation}
By $L^1$-contraction for solutions to the scalar viscous conservation laws \cite{DS}, we have 
\begin {equation}\label{I8}
   \begin{aligned}
   \|I_1(t,\cdot)\|_{ L^1(\mathbb{R})}\leq \|u_0(\cdot)-U(\cdot-x_0)\|_{L^1(\RR)}.
     \end{aligned}
\end{equation}
Inspired by \cite{KV}, we denote $\tau=\tau(t):=\mathbf{X}(t)-x_0$ and
\begin {equation}\label{I9}
\|I_2(t,\cdot)\|^2_{ L^2(\mathbb{R})}=\int_{\RR}|U(\x-\tau)-U(\x)|^2d\x:=R(\tau).\notag
\end {equation}
Then we have
\begin {equation*}
   \begin{array}{ll}
\di \frac{\partial{R}}{\partial \tau}(\tau)&\di=-2\int_{\RR}[U(\x-\tau)-U(\x)]\frac{\partial U(\x-\tau)}{\partial \x} d\x\\[3mm]
&\di =-2\int_{\RR}\int_{\x}^{\x-\tau}\frac{\partial U(y)}{\partial y}dy\frac{\partial U(\x-\tau)}{\partial \x} d\x\\[4mm]
&\di =2\int_{\RR}\int_{x}^{x+\tau}(U)^\prime(y)(U)^\prime(x)dydx,
 \end{array}
\end{equation*}
which is strictly positive for $\tau>0$.
Furthermore, if $\tau>1$, then we have
$$
\frac{\partial{R}}{\partial \tau}(\tau)\geq2\int_{\RR}\int_x^{x+1}(U)^\prime(y)(U)^\prime(x)dydx:=\beta>0.
$$
Hence, for $\tau>1$
$$
R(\tau)\geq R(1)+\beta(\tau-1)\geq\beta(\tau-1),
$$
and then it holds that
$$
1<\tau\leq\frac{R(\tau)}{\beta}+1.
$$ 
Since $\di R(\tau)=\int_{\RR}|U(\x-\tau)-U(\x)|^2d\x=\int_{\RR}|U(\x)-U(\x+\tau)|^2d\x=R(-\tau)$,  that is,  $R(\tau)$ is an even function with respect to $\tau\in \mathbb{R}$, we have $\forall \tau<-1,$
   $$
R(\tau)\geq R(1)+\beta(-\tau-1)=R(1)+\beta(|\tau|-1)\geq \beta(|\tau|-1).
$$
On the other hand, it is obvious that for $\tau\in[-1,1]$,
$$
R(\tau)\geq 0>\beta(|\tau|-1).
$$
Therefore, we have $\forall \tau \in \RR$, 
\begin {equation}\label{I10}
|\tau|\leq \frac{R(\tau)}{\beta}+1,\notag
\end {equation}
where $\di \beta:=2\int_{\RR}\int_x^{x+1}(U)^\prime(y)(U)^\prime(x)dydx>0.$
Equivalently, it holds that $\forall t \in \RR^+$, 
\begin {equation}\label{I11}
|\mathbf{X}(t)-x_0|\leq\frac{ \|I_2(t,\cdot)\|^2_{ L^2(\mathbb{R})}}{\beta}+1.
\end {equation}
Using the inequality $(a+b)^2\geq a^2-2|ab|$ and \eqref{3-2}, \eqref{I8}, we can get
\begin {equation}\label{I13}
   \begin{aligned}
  & \|I_2(t,\cdot)\|^2_{ L^2(\mathbb{R})}\leq 2\int_{\RR}|I_1(t,x)||I_2(t,x)|dx+\int_{\RR}\big(I_1(t,x)+I_2(t,x)\big)^2dx\\
  & \leq 2\|I_2(t,\cdot)\|_{ L^\infty(\mathbb{R})}\|I_1(t,\cdot)\|_{ L^1(\mathbb{R})}+\int_{\RR}|u(t,x)-U(x-s t - \mathbf{X}(t))|^2dx\\
  & = 2\|I_2(t,\cdot)\|_{ L^\infty(\mathbb{R})}\|I_1(t,\cdot)\|_{ L^1(\mathbb{R})}+\int_{\RR}\phi^2(t,\xi)d\xi\\
   &\leq 2|u_--u_+|\|u_0(\cdot)-U(\cdot-x_0)\|_{L^1(\RR)}+C_0\|\phi_0\|^2_{H^1(\RR)}.
\end{aligned}
\end{equation}
where in the last inequality we use the fact that $\|I_2(t,\cdot)\|_{ L^\infty(\mathbb{R})}\leq|u_--u_+|$.
By \eqref{I11} and \eqref{I13}, we get that $\mathbf{X}(t)-x_0$ is uniformly bounded and satisfies 
\begin {equation}\label{I14}
  |\mathbf{X}(t)-x_0|\leq \frac{1}{\beta} \Big(2|u_--u_+|\|u_0(\cdot)-U(\cdot-x_0)\|_{L^1(\RR)}+C_0\|\phi_0\|_{H^1(\RR)}\Big)+1.
\end{equation}
Since $U(x)$ is a monotone decreasing function, we have
\begin {equation*}
\begin{aligned}
 \|I_2(t,\cdot)\|_{ L^1(\mathbb{R})}&=sgn(\tau(t)) \int_{\RR}\big(U(\x-\tau(t))-U(\x)\big)d\x\\
& =sgn(\tau(t))  \int_{\RR}\int_0^{-\tau(t)} \partial_yU(\x+y)dyd\x\\
&=-\int_{\RR}\int_0^{|\tau(t)|} \partial_yU(\x+y)dyd\x.
  \end{aligned}
\end{equation*}
Then, using Fubini's theorem to get
\begin {equation}\label{I15}
\|I_2(t,\cdot)\|_{ L^1(\mathbb{R})}=|\mathbf{X}(t)-x_0||u_--u_+|.
\end {equation}
By \eqref{I8}, \eqref{I14} and \eqref{I15}, we have 
\begin {equation}\label{I16}
   \begin{aligned}
\|\phi(t,\cdot)\|_{ L^1(\mathbb{R})}= \|(I_1+I_2)(t,\cdot)\|_{ L^1(\mathbb{R})}\leq \|I_1(t,\cdot)\|_{ L^1(\mathbb{R})}+\|I_2(t,\cdot)\|_{ L^1(\mathbb{R})}\leq C_1,
\end{aligned}
\end{equation}
where $$
\begin{array}{ll}
\di C_1=\|u_0(\cdot)-U(\cdot-x_0)\|_{L^1(\RR)}+|u_--u_+|\Big[\frac{1}{\beta} \Big(\di 2|u_--u_+|\|u_0(\cdot)-U(\cdot-x_0)\|_{L^1(\RR)}\\[3mm]
\di \hspace{8.5cm} +C_0\|\phi_0\|_{H^1(\RR)}\Big)+1\Big].
\end{array}$$
 Then Gagliardo-Nirenberg intepolation inequality   shows that
 $$\|\phi(t,\cdot)\|_{ L^2(\mathbb{R})}\leq C \|\phi(t,\cdot)\|^{\frac{2}{3}}_{ L^1(\mathbb{R})} \|\partial_x \phi(t,\cdot)\|^{\frac{1}{3}}_{ L^2(\mathbb{R})}.$$ 
By \eqref{I16}, we have 
  \begin {equation}\label{I17} 
  \begin{aligned}
 \|\phi(t,\cdot)\|^3_{ L^2(\mathbb{R})} \leq C \|\phi(t,\cdot)\|^2_{ L^1(\mathbb{R})} \|\partial_x\phi(t,\cdot)\|_{ L^2(\mathbb{R})}\leq C_2\|\partial_x\phi(t,\cdot)\|_{ L^2(\mathbb{R})}.
 \end{aligned}
  \end {equation}
 From the proof of Proposition \ref{P2.1} and  Proposition \ref{P2.1++} (see \eqref{2.31+}) and \eqref{2.16++},  for positive constant $C$ we have
  \begin {equation}\label{-17+} 
  \frac{d}{dt} \int_{\RR}(Ca+1)(\phi^{\mb{X}})^2d\xi \leq -2\int_{\RR}(\phi^{\mb{X}}_\xi)^2d \xi
 \end {equation}
 Since  $0<c< a^{-\mb{X}}< C$ for positive constants $c$ and $C$, by \eqref{I17} and  \eqref{-17+}, we have 
 \begin {equation}\label{-18} 
  \begin{aligned}
 \frac{d}{dt} \|\phi(t,\cdot)\sqrt {Ca^{-\mb{X}}+1}\|^2_{ L^2(\mathbb{R})}&\leq -2\|\partial_x\phi(t,\cdot)\|^2_{ L^2(\mathbb{R})}\nonumber\\
&\leq -\frac{2}{(C_2)^2}  \|\phi(t,\cdot)\|^6_{ L^2(\mathbb{R})}  \leq -C_3\|\phi(t,\cdot)\sqrt {Ca^{-\mb{X}}+1}\|^6_{ L^2(\mathbb{R})}\notag,
 \end{aligned}
  \end {equation} 
 which implies the decay estimate 
    \begin {equation}\label{-19} 
  \|\phi(t,\cdot)\sqrt {Ca^{-\mb{X}}+1}\|^4_{ L^2(\mathbb{R})}\leq\frac{\|\phi_0\sqrt {Ca^{-\mb{X}}+1}\|^4_{ L^2(\mathbb{R})}}{1+2C_3t\|\phi_0\sqrt {Ca^{-\mb{X}}+1}\|^4_{ L^2(\mathbb{R})}}.
  \end {equation}
  By \eqref{-19} and  the inequality $2(a+b)^\frac{1}{4}\geq a^\frac{1}{4}+b^\frac{1}{4}$, we have
    \begin {equation}\label{-20} 
  \|\phi(t,\cdot)\|_{ L^2(\mathbb{R})}\leq\frac{C_*\|\phi_0\|_{ L^2(\mathbb{R})}}{1+C_*t^\frac{1}{4}\|\phi_0\|_{ L^2(\mathbb{R})}},
  \end {equation}
 for some uniform-in-time positive constant $C_*$. 
 Thus  we complete the proof of Theorem \ref{T1.1}.

\section{Uniform-in-time a priori estimates}
In this section, we prove Proposition \ref{T3.2} for the uniform-in-time a priori estimates. To do this, we assume that the Cauchy problem \eqref{2.3} has a solution $ \phi\in C([0,T];H^1(\RR))\cap L^2(0,T; H^2(\RR))$ for some constant $T>0$. First we  prove $L^2$ relative entropy estimate for $\phi$ as follows. 
 \begin{proposition}\label{P2.1} There exists a positive constant $\epsilon_1>0$, such that if $\mathcal{ N}(T):=\displaystyle \sup_{0\leq t\leq T}\|\phi(t,\cdot)\|_{H^1(\RR)} \leq\epsilon_1$, then  $\forall t\in[0,T],$ it holds that
 \begin{align}\label{2.9}
&\di  \frac{d}{dt} \int_{\RR}(\phi^{\mb{X}})^2  a d\xi\leq 0,\\
&\di \|\phi(t,\cdot)\|_{ L^2(\mathbb{R})}^2
+\int_0^t\int_\mathbb{R} \big|U^{-\mb{X}}_\xi\big| {\phi}^2d \xi d\tau+ \int_0^t|  \mathbf{ \dot  X}(\tau)|^2d\tau\leq C\|\phi_0\|_{ L^2(\mathbb{R})}^2,
\end{align}
where the positive constant $C$ is independent of $T$.
  \end{proposition}
  In order to prove Proposition \ref{P2.1}, we need to use a Poincar${\acute{\rm e}}$ type inequality and weighted energy method. Thus, we start with a Poincar${\acute{\rm e}}$ type inequality.
\begin{lemma}\label{L2.1} (\cite{VKM}) For any $f: [0,1]\rightarrow\mathbb{R} $ satisfying $\di \int_0^1 y(1-y)|f^\prime(y)|^2dy<+\infty,$ it holds that
  \begin{equation}\label{2.10}
\int_0^1 \Big|f-\int_0^1 fdy\Big|^2dy\leq\frac{1}{2}\int_0^1y(1-y)|f^\prime(y)|^2dy,
  \end{equation}
  that is,
    \begin{equation}\label{2.11}
\int_0^1 f^2(y)dy-\Big(\int_0^1 f(y)dy\Big)^2\leq\frac{1}{2}\int_0^1y(1-y)|f^\prime(y)|^2dy.
   \end{equation} 
  \end{lemma}

\

 For simplicity, we denote $a^{-\mb X}:=a(U^{-\mb{X}}(\x))=a\big(U(\x-\mb{X}(t))\big)$.
Multiplying the equation $\eqref{2.3}_1$ by $ a^{-\mb X}\phi$,  we can get
 \begin{equation}\label{2.12}
 \begin{aligned}
 \left[\frac{1}{2}\phi^2 a^{-\mb{X}}\right]_t&-s\phi_\xi\phi a^{-\mb{X}}+\big[f(\phi+U^{-\mb{X}})-f(U^{-\mb{X}} )\big]_\xi \phi a^{-\mb{X}}\\[2mm]
 &+\frac{1}{2}\mathbf{ \dot  X}(t)\phi^2 (a^{-\mb{X}})_\xi-\mathbf{\dot X}(t)U^{-\mb{X}}_\xi\phi a^{-\mb{X}}= \phi_{\xi\xi}\phi a^{-\mb{X}}.
 \end{aligned}
\end{equation}
Integrating \eqref{2.12} over $\RR$ with respect to $\xi$ and changing variable $\xi\rightarrow \xi-\mathbf{  X}(t)$ and denoting $\phi^{\mb{X}}:=\phi(t,\xi+\mathbf{ X}(t))$, we have
 \begin{equation}\label{2.13}
 \begin{aligned}
& \left[\frac{1}{2}\int_{\RR}(\phi^{\mb{X}})^2a d\xi\right]_t+\int_{\RR}(\phi^{\mb{X}})^2\left[\frac{s}{2}a_\x-\frac{1}{2}a_{\x\x}-\frac{a_\x}{2}f^\prime(U)+\frac{a f^{\prime\prime}(U)}{2}U_\x\right]d\x  \\[2mm]
 &\qquad \qquad\qquad  +\mathbf{ \dot  X}(t)\left[\frac{1}{2}\int_{\RR}(\phi^{\mb{X}})^2a_\xi d\x -\int_{\RR}\phi^{\mb{X}} a U_\xi d\x\right]\\[2mm]
&\qquad \qquad\qquad  +\int_{\RR} a(\phi^{\mb{X}}_\x)^2d\x+ \int_{\RR}O(1)(\phi^{\mb{X}})^3 U_\x d\x=0.
 \end{aligned}
\end{equation}
where we have used the fact
 \begin{equation}\label{2.14}
 \begin{aligned}
& \int_{\RR}  a^{-\mb{X}} \phi [f(\phi+U^{-\mb{X}})-f(U^{-\mb{X}} )]_\xi d\x\\
& = \int_{\RR} a \phi^{\mb{X}} [f(\phi^{\mb{X}}+U)-f(U )]_\xi d\x
\\& = \int_{\RR} a U_\x[f(\phi^{\mb{X}}+U)-f(U)-f^\prime(U)\phi^{\mb{X}}]d\x \\
&\quad +\int_{\RR} a_\x\Big[\int_U^{\phi^{\mb{X}}+U}f(\eta)d\eta-\phi^{\mb{X}}f(\phi^{\mb{X}}+U)\Big]d\x\\
& = \int_{\RR} (\phi^{\mb{X}})^2U_\x\Big[\frac{a f^{\prime\prime}(U)}{2}-\frac{a^\prime f^\prime(U)}{2}\Big]d \x+\int_{\RR} O(1)(\phi^{\mb{X}})^3 U_\x d\x.
  \end{aligned}
\end{equation}

 Let  $$y:=\frac{U(\xi)-u_-}{u_+-u_-},$$ then we have $\xi\in(-\infty,+\infty)\iff y \in(0,1)$. Since $y_\x=\frac{U_\x}{u_+-u_-}>0$,  there exists a unique inverse function $\x=\x(y)$ by the inverse function theorem. For any fixed $t>0$, we denote  
 \begin{equation}
 \psi(t,y):= \phi^{\mb{X}}(t,\x(y)) a(U(\x(y))).
 \end{equation}
In order to use the weighted Poincar${\acute{\rm e}}$ inequality with $\psi(t,y)$, we first have
\begin{equation}\label{2.15}
 \begin{aligned}
-\frac{{\mathbf{ \dot  X}(t)}}{2}\int_{\RR}\phi^{\mb{X}} a U_\xi d\x=2\Big(\int_{\RR} \phi^{\mb{X}} a \frac{U_\xi}{u_+-u_-}d\xi\Big)^2= 2\Big(\int_0^1\psi(t,y) dy\Big)^2\\
\geq2 \int_0^1 \psi^2(t,y)dy-\int_0^1 |\psi_y(t,y)|^2y(1-y)dy.
 \end{aligned}
\end{equation}
Furthermore, we have 
\begin{equation}\label{2.16}
2\int_0^1\psi^2(t,y)dy=\frac{2}{u_+-u_-}\int_{\RR} (\phi^{\mb{X}})^2a^2U_\xi d\xi,
\end{equation}
and 
\begin{equation}\label{2.17}
 \begin{aligned}
\di \int_0^1   |\psi_y(t,y)|^2y(1-y)dy&\di =\frac{1}{u_+-u_-} \int_{\RR}\big[(\phi^{\mb{X}} a)_\x\big]^2\frac{(U-u_-)(u_+-U)}{U_\xi}d\xi\\[3mm]
\di 
 &\di =\int_{\RR}\frac{[(\phi^{\mb{X}} a)_\x] ^2}{ a}d\xi \\
 &\di =\int_{\RR} a (\phi^{\mb{X}}_\x)^2d\x-\int_{\RR}\Big(a_{\x\x}-\frac{ a_\x^2}{a}\Big)(\phi^{\mb{X}})^2d\x.
 \end{aligned}
\end{equation}
By \eqref{2.15}-\eqref{2.17}, we have 
\begin{equation}\label{2.18}
 \begin{aligned}
-\frac{{\mathbf{ \dot  X}(t)}}{2}\int_{\RR}\phi^{\mb{X}} a U_\xi d\x+ \int_{\RR}a \big(\phi^{\mb{X}}_\x\big)^2d\x\geq  \int_{\RR}\left[\frac{2a^2U_\x}{u_+-u_-}+ a_{\x\x}-\frac{ a_\x^2}{a}\right](\phi^{\mb{X}})^2d\x.
 \end{aligned}
\end{equation}
By \eqref{2.4}, we have $a_\x=a^\prime(U)U_\x$ and $a_{\x\x}=a^\prime(U)(f^\prime(U)-s)U_\x+a^{\prime\prime}(U)U_\x^2$. Then from \eqref{2.13} and \eqref{2.18}, we can get
\begin{equation}\label{2.19}
 \begin{aligned}
& \left[\frac{1}{2}\int_\RR(\phi^{\mb{X}})^2 ad\x\right]_t+\int_{\RR}(\phi^{\mb{X}})^2U_\x\bigg[\frac{2a^2}{u_+-u_-}+\frac{a f^{\prime\prime}(U)}{2}+\Big({\frac{a^{\prime\prime}}{2}}-\frac{(a^\prime)^2}{a}\Big)h(U)\bigg]d\x  \\[2mm]
&\quad +\frac{1}{2}\mathbf{ \dot  X}(t)\int_{\RR}(\phi^{\mb{X}})^2a_\xi d\x+\frac{(u_--u_+)^2}{8}|{\mathbf{ \dot  X}(t)}|^2+\int_{\RR}O(1)\big|(\phi^{\mb{X}})^3 U_\x\big| d\x \leq 0,
 \end{aligned}
\end{equation}
where we have used the fact 
$$
-\frac{{\mathbf{ \dot  X}(t)}}{2}\int_{\RR}\phi^{\mb{X}}a U_\xi d\x=\frac{(u_--u_+)^2}{8}|{\mathbf{ \dot  X}(t)}|^2.
$$
Note that

\begin{equation}\label{2.20}
g(U)=\frac{2a^2}{u_+-u_-}+\frac{ af^{\prime\prime}(U)}{2}+\left[\frac{ a^{\prime\prime}}{2}-\frac{(a^\prime)^2}{a}\right]h(U).  
\end{equation}  
By direct calculations, we can obtain
\begin{equation}\label{2.21}
\begin{array}{ll}
\di g(U)=-\frac{|u_+-u_-|}{(U-u_-)^2(u_+-U)^2}\Big[(U-u_-)(u_+-U)\big(hh^{\prime\prime}-(h^\prime)^2\big)\\[4mm]
\di \hspace{6cm} +3h^2+hh^\prime(u_++u_--2U)\Big]\\
\di \qquad\quad :=-\frac{|u_+-u_-|}{(U-u_-)^2(u_+-U)^2} m(U),
\end{array}
\end{equation}  
where
\begin{equation}\label{m}
m(U):=(U-u_-)(u_+-U)(hh^{\prime\prime}-(h^\prime)^2)+3h^2+hh^\prime(u_++u_--2U).
\end{equation}
 and $h=h(U)$ is defined in \eqref{1.4}.



\begin{lemma}\label{L2.2} There exist positive constants $\beta=\beta(u_+,u_-,p )$ such that $\forall U\in[u_+,u_-],$ it holds that $g(U)\leq -\beta<0$.
\end{lemma} 
\begin{proof}
 
By \eqref{2.4}, we have
$$a(u_+)=s-f^\prime(u_+)=\frac{f(u_+)-f(u_-)}{u_+-u_-}-f^\prime(u_+)>0.$$
By \eqref{2.20}, we have 
\begin{equation}\label{2.22}
 \begin{aligned}
g(u_+)&=2 a(u_+)\Big[\frac{f^{\prime\prime}(u_+)}{4}+\frac{ a(u_+)}{u_+-u_-}\Big]\\
&=\frac{2 a(u_+)}{u_--u_+}\Big[\frac{f^{\prime\prime}(u_+)(u_--u_+)}{4}+f^\prime(u_+)-\frac{f(u_+)-f(u_-)}{u_+-u_-}\Big]\\
&=-\frac{2 a(u_+)}{(u_--u_+)^2}\Big[f(u_-)-f(u_+)-f^\prime(u_+)(u_--u_+)-\frac{f^{\prime\prime}(u_+)}{4}(u_--u_+)^2\Big]\\
&=-\frac{2 a(u_+)}{(u_--u_+)^2}\Big[\frac{f^{\prime\prime}(u_+)}{4}(u_--u_+)^2+\frac{f^{\prime\prime\prime}(\theta)}{6}(u_--u_+)^3\Big]<0,
 \end{aligned}
\end{equation}
where $\theta\in(u_+,u_-)$.
On the other hand, by \eqref{2.4},
$$a(u_-)=f^\prime(u_-)-s=f^\prime(u_-)-\frac{f(u_+)-f(u_-)}{u_+-u_-}>0,$$
 then we have
\begin{equation}\label{2.23}
 \begin{aligned}
g(u_-)&=2a(u_-)\Big[\frac{f^{\prime\prime}(u_-)}{4}+\frac{ a(u_-)}{u_+-u_-}\Big]\\
&=2a(u_-)\left[u_-^{p -2}\Big(\frac{p ^2-5p +4}{4}+\frac{(2-p )u_+}{u_--u_+}+\frac{u_+^2}{(u_--u_+)^2}\Big)-\frac{u_+^p }{(u_--u_+)^2}\right]\\
&\leq 2 a(u_-)\left[u_-^{p -2}\Big(\frac{(2-p )u_+}{u_--u_+}+\frac{u_+^2}{(u_--u_+)^2}\Big)-\frac{u_+^p }{(u_--u_+)^2}\right]\\
&=\frac{2a(u_-)}{(u_--u_+)^2}\underbrace{\Big[(2-p )u_+(u_--u_+)u_-^{p -2}+u_+^2\big(u_-^{p -2}-u_+^{p-2}\big) \Big]}_{:=l_1(u_-,u_+)}<0,
 \end{aligned}
\end{equation}
where we have used the fact that $\forall p\in [2, 4], \forall 0<u_+<u_-, $
$$l_1(u_-,u_+)<0.$$

Next, we claim that 
\begin{equation}\label{claim}
\forall U\in(u_+, u_-),\quad  g(U)<0.
\end{equation}

To prove the claim \eqref{claim}, by \eqref{m}, we first have
\begin{equation}\label{2.24}
 \begin{aligned}
m(U)=\underbrace{[h+(u_--U)h^\prime][h+(u_+-U)h^\prime]}_{I_1}+\underbrace{(U-u_-)(u_+-U)hh^{\prime\prime}+2h^2}_{I_2}.
\end{aligned}
\end{equation}
First we estimate $I_1$. Since $\forall U\in(u_+,u_-)$, we have
\begin{equation}\label{2.25}
 \begin{aligned}
h(U)+(u_--U)h^\prime(U)&=-[h(u_-)-h(U)-(u_--U)h^\prime(U)]\\
&=-\frac{h^{\prime\prime}(U)}{2}(u_--U)^2-\frac{h^{\prime\prime\prime}(\bar U)}{6}(u_--U)^3\\
&\leq-\frac{h^{\prime\prime}(U)}{2}(u_--U)^2<0, 
\end{aligned}
\end{equation}
where $\bar U\in[U,u_-]$ and then $h^{\prime\prime\prime}(\bar U)=f^{\prime\prime\prime}(\bar U)>0, \forall p \in[2,4]$. 


On the other hand, we can calculate that
\begin{equation}\label{2.26}
 \begin{aligned}
&h(U)+(u_+-U)h^\prime(U)+\frac{h^{\prime\prime}(U)}{4}(u_+-U)^2\\
&=f^\prime(U)(u_+-U)+f(U)-f(u_+)+\frac{f^{\prime\prime}(U)}{4}(u_+-U)^2\\[3mm]
&=u_+^p \left[\frac{(p -1)(p -4)}{4}{\Big(\frac{U}{u_+}\Big)}^p +\frac{3p -p ^2}{2}{\Big(\frac{U}{u_+}\Big)}^{p -1}+\frac{p (p -1)}{4}{\Big(\frac{U}{u_+}\Big)}^{p -2}-1\right].
\end{aligned}
\end{equation}

Set $z:=\frac{U}{u_+}$.  Then we have $U\in(u_+,u_-)\iff z\in(1,\frac{u_-}{u_+})$. Define
$$H(z):=\frac{(p -1)(p -4)}{4}z^{p }+\frac{3p -p ^2}{2}z^{p -1}+\frac{p (p -1)}{4}z^{p -2}-1.$$
It is obvious to compute that 
$$
H^\prime(z)=\frac{p (p -1)}{4}z^{p -3}\Big[(p -4)z^2+2(3-p )z+(p -2)\Big].
$$
 Since $\forall z\in(1,\frac{u_-}{u_+}), $ it holds that $(p -4)z^2+2(3-p )z+(p -2)<0$. Hence, $\forall p \in[2,4]$, we can get $H^\prime(z)<0, \forall z\in(1,\frac{u_-}{u_+})$. Therefore, we have $H(z)<H(1)=0, \forall z\in(1,\frac{u_-}{u_+}),$ which and \eqref{2.26} imply that  $\forall p \in[2,4]$,
\begin{equation}\label{2.27}
 \begin{aligned}
h(U)+(u_+-U)h^\prime(U)<-\frac{h^{\prime\prime}(U)}{4}(u_+-U)^2<0,\quad  \forall U\in(u_+,u_- ).
\end{aligned}
\end{equation} 

Therefore, by \eqref{2.25} and \eqref{2.27}, we have $\forall p \in[2,4]$,  $\forall U\in(u_+,u_-),$
\begin{equation}\label{2.28}
I_1=[h+(u_--U)h^\prime][h+(u_+-U)h^\prime]>\frac{(h^{\prime\prime})^2}{8}(u_+-U)^2(u_--U)^2.
\end{equation} 
By \eqref{2.24} and \eqref{2.28}, we have $\forall p \in[2,4]$,
\begin{equation}\label{2.29}
 \begin{aligned}
m(U)=I_1+I_2>\frac{(h^{\prime\prime})^2}{8}(u_+-U)^2(u_--U)^2+(U-u_-)(u_+-U)hh^{\prime\prime}+2h^2,\\
=2\left[h+\frac{h^{\prime\prime}(U)(U-u_-)(u_+-U)}{4}\right]^2> 0, \quad \forall U\in(u_+,u_- ).\ \ 
\end{aligned}
\end{equation} 
Hence,  by \eqref{2.21}, we proved the claim \eqref{claim}. From \eqref{2.22} and \eqref{2.23}, we have $g(u_\pm)<0$. By the properties of continuous function over closed interval, there exists a positive constant $\beta=\beta(u_-, u_+, p)$, such that $$g(U)\leq -\beta, \quad \forall U\in[u_+,u_-].$$

\end{proof}

Thus, by \eqref{2.19}, we have
\begin{equation}\label{2.31+}
 \begin{aligned}
& \big[\frac{1}{2}\int_{\RR}(\phi^{\mb{X}})^2  a d\xi \big]_t+ \beta\int_{\RR}(\phi^{\mb{X}})^2|U_\x| d\x 
+\frac{1}{2}\mathbf{ \dot  X}(t)\int_{\RR}(\phi^{\mb{X}})^2a_\xi d\x
\\
&+\frac{(u_--u_+)^2}{8}|{\mathbf{ \dot  X}(t)}|^2+\int_{\RR}O(1)(\phi^{\mb{X}})^3 U_\x d\x \leq 0.
 \end{aligned}
\end{equation} 

On the other hand, we have 
 \begin{equation}\label{2.32}
  \begin{aligned}
&\left|\frac{1}{2}\mathbf{ \dot  X}(t)\int_{\RR}(\phi^{\mb{X}})^2a_\xi d\x\right|\\
&\leq \frac{4}{(u_--u_+)^2}\|\phi^{\mb{X}}\|_{L^\infty(\RR)}\|a\|_{L^\infty(\RR)}\|a^\prime\|_{L^\infty(\RR)}\int_{\RR} |U_\x| d\x \int_{\mathbb{R}} (\phi^{\mb{X}})^2|U_\xi | d\xi\\
&\leq C\|\phi\|_{H^1(\RR)} \int_{\mathbb{R}} (\phi^{\mb{X}})^2|U_\xi|  d\xi\leq C\epsilon_1 \int_{\mathbb{R}} (\phi^{\mb{X}})^2|U_\xi|  d\xi.
 \end{aligned}
\end{equation}
By \eqref{2.31+} and \eqref{2.32}, we have 
 \begin{equation}\label{2.32+}
  \begin{aligned}
& \left[\frac{1}{2}\int_{\RR}(\phi^{\mb{X}})^2  a d\xi\right]_t + \Big(\beta-C\epsilon_1\Big) \int_{\RR}(\phi^{\mb{X}})^2|U_\x| d\x+\frac{(u_--u_+)^2}{8}|{\mathbf{ \dot  X}(t)}|^2 \leq 0.
 \end{aligned}
\end{equation}

 Integrating \eqref{2.32+} with respect to $t$, then changing of variable $\x\rightarrow \x-\mathbf{ X}(t)$, and choosing the suitable smallness of $\epsilon_1$, we can prove Proposition \ref{P2.1}.
 
  \begin{proposition}\label{P2.1++} There exists a positive constant $\epsilon_2>0$, such that if $\mathcal{ N}(T):=\displaystyle \sup_{0\leq t\leq T}\|\phi(t,\cdot)\|_{H^1(\RR)} \leq\epsilon_2$, then  $\forall t\in[0,T]$, it holds that
 \begin{align}\label{2.9++}
&\di \|\phi(t,\cdot)\|_{ L^2(\mathbb{R})}^2+\int_0^t\|\phi_\xi(\tau, \cdot)\|_{ L^2(\mathbb{R})}^2d\tau\nonumber\\
&\di \quad 
+\int_0^t\int_\mathbb{R} \big|U^{-\mb{X}}_\xi\big| {\phi}^2d \xi d\tau+ \int_0^t|  \mathbf{ \dot  X}(\tau)|^2d\tau\leq C\|\phi_0\|_{ L^2(\mathbb{R})}^2,
\end{align}
where the positive constant $C$ is independent of $T$.
  \end{proposition}
  \begin{proof}
 Multiplying the equation $\eqref{2.3}_1$ by $ \phi$,  we can get
 \begin{equation}\label{2.12++}
 \begin{aligned}
 \left(\frac{\phi^2}{2} \right)_t&-s\phi_\xi\phi +\big[f(\phi+U^{-\mb{X}})-f(U^{-\mb{X}} )\big]_\xi \phi -\mathbf{\dot X}(t)U^{-\mb{X}}_\xi\phi = \phi_{\xi\xi}\phi.
 \end{aligned}
\end{equation}
Integrating \eqref{2.12++} over $\RR$ with respect to $\xi$ and changing variable $\xi\rightarrow \xi-\mathbf{  X}(t)$ and denoting $\phi^{\mb{X}}:=\phi(t,\xi+\mathbf{ X}(t))$, we have
 \begin{equation}\label{2.13++}
 \begin{aligned}
& \left[\frac{1}{2}\int_{\RR}(\phi^{\mb{X}})^2d \xi\right]_t+\int_{\RR}(\phi^{\mb{X}})^2\frac{f^{\prime\prime}(U)}{2}U_\x d\x  -\mathbf{ \dot  X}(t) \int_{\RR}\phi^{\mb{X}} U_\xi d\x\\[2mm]
&\qquad \qquad\qquad  +\int_{\RR} (\phi^{\mb{X}}_\x)^2d\x+ \int_{\RR}O(1)(\phi^{\mb{X}})^3 U_\x d\x=0.
 \end{aligned}
\end{equation}
where we have used the fact
 \begin{equation}\label{2.14++}
 \begin{aligned}
& \int_{\RR}   \phi [f(\phi+U^{-\mb{X}})-f(U^{-\mb{X}} )]_\xi d\x\\
& = \int_{\RR}  [f(\phi^{\mb{X}}+U)-f(U )]_\xi d\x
\\& = \int_{\RR} U_\x[f(\phi^{\mb{X}}+U)-f(U)-f^\prime(U)\phi^{\mb{X}}]d\x \\
& = \int_{\RR} (\phi^{\mb{X}})^2U_\x\frac{ f^{\prime\prime}(U)}{2}+ \int_{\RR}O(1)(\phi^{\mb{X}})^3 U_\x d\x.
  \end{aligned}
\end{equation}
By Cauchy inequality and H${ \ddot{\rm o}}$lder inequality, we have 
 \begin{equation}\label{2.15++}
 \begin{aligned}
 -\mathbf{ \dot  X}(t) \int_{\RR}\phi^{\mb{X}} U_\xi d\x&\leq\frac{|\mathbf{ \dot  X}(t)|^2 }{2}+\frac{\left(\int_{\RR}\phi^{\mb{X}} U_\xi d\x\right)^2}{2}\\
& \leq \frac{|\mathbf{ \dot  X}(t)|^2 }{2}+ \frac{|u_--u_+|}{2}\int_{\RR}(\phi^{\mb{X}})^2 U_\xi d\x
  \end{aligned}
\end{equation}
 Hence, it holds that
  \begin{equation}\label{2.16++}
 \begin{aligned}
  \left[\frac{1}{2}\int_{\RR}(\phi^{\mb{X}})^2 d\xi \right]_t+\int_{\RR} (\phi^{\mb{X}}_\x)^2d\x\leq C\left[\int_{\RR}(\phi^{\mb{X}})^2 |U_\xi|d\x+|\mathbf{ \dot  X}(t)|^2\right].
  \end{aligned}
\end{equation}
Integrating \eqref{2.16++} with respect to $\xi$, and  combining with  Proposition \ref{P2.1} we complete the proof.
 \end{proof}

 \begin{proposition}\label{P2.2} 
There exists a positive constant $\epsilon_3>0$, such that if $\mathcal{ N}(T):=\displaystyle \sup_{0\leq t\leq T}\|\phi(t,\cdot)\|_{H^1(\RR)} \leq\epsilon_3$, 
 then there exists a uniform-in-time positive constant $C$ such that for all $t\in[0,T]$, it holds that
 \begin{equation}\label{22}
\|\phi_\xi (t,\cdot)\|^2+\int_0^t\|\phi_{\xi\xi}(\tau, \cdot)\|^2d\tau\leq C\|\phi_0\|_{H^1(\mathbb{R})}^2.
\end{equation}
    \end {proposition}
  \begin{proof} Multiplying \eqref{2.3} by $-\phi_{\xi\xi}$ and integrating the resulted equation with respect to $\xi$, we can obtain
   \begin{equation}\label{2.33}
 \frac{1}{2}\frac{d}{dt}\| \phi_\xi\|^2+\|\phi_{\xi\xi}\|^2  -\mathbf{ \dot X}(t)\int_{\mathbb{R}}\phi_{\xi\xi}U^{-\mb{X}}_\xi d\xi=\int_{\mathbb{R}}\Big(f(\phi+U^{-\mb{X}})-f(U^{-\mb{X}})\Big)_\xi \phi_{\xi\xi}d\xi.
  \end{equation}
  On one hand, we have 
 \begin{equation}\label{2.34}
 \begin{aligned}
&  \int_{\mathbb{R}}\Big|\Big(f(\phi+U^{-\mb{X}})-f(U^{-\mb{X}})\Big)_\xi\Big|^2d\xi \\
& \leq 2\int_{\mathbb{R}} \Big|\Big(f^\prime(\phi+U^{-\mb{X}})-f^\prime(U^{-\mb{X}})\Big)U^{-\mb{X}}_\xi\Big|^2d\xi+ 2\int_{\mathbb{R}} \Big|f^\prime(\phi+U^{-\mb{X}})\phi_\xi \Big|^2 d\xi\\
 &\leq C \int_{\mathbb{R}}\phi^2 (U^{-\mb{X}}_\xi)^2d\xi+C\|\phi_\xi\|^2 \leq C\int_{\mathbb{R}}\phi^2 |U^{-\mb{X}}_\xi |d\xi+C\|\phi_\xi\|^2.
   \end{aligned}
    \end{equation}
On the other hand, we have 
 \begin{equation}\label{2.35}
 \begin{aligned}
\Big|\mathbf{ \dot X}(t)\int_{\mathbb{R}}\phi_{\xi\xi}U^{-\mb{X}}_\xi d\xi \Big| \leq \frac{1}{2}\|\phi_{\xi\xi}\|^2+ C|\mathbf{ \dot X}(t)|^2.
   \end{aligned}
    \end{equation}
By Proposition \ref{P2.1++} and \eqref{2.33}-\eqref{2.35}, we prove Proposition \ref{P2.2}.
\end{proof}
Therefore, Proposition \ref{P2.1++} and Proposition \ref{P2.2} complete the proof of Proposition \ref{T3.2} by .

 \section{Appendix}
{ \bf The proof of Proposition  \ref{2.1}}. Since both the non-local equation \eqref{pe+} for  $\phi^{\mb X}$ and the initial value $\phi^{\mb X}_0(\xi)$ in \eqref{in} are independent of the definition of the shift function $\mb X(t)$, we will omit the superscript  $\mb X$ for simplicity in this subsection. For any initial time $\tau\geq 0$, rewrite the Cauchy problem \eqref{aaa}  as
 \begin{equation}\label{-aaa}
\left\{\begin{array}{l}
\displaystyle\phi_t-s \phi_{\x}+\frac{4(\phi_\xi+U_\x)}{(u_--u_+)^2}\int_{\mathbb{R}} a(U) U_\xi \phi  d\xi+\big[f(\phi+ U)-f(U)\big]_\x  =\phi_{\x\x},\quad t>\tau,\\[3mm]
 \phi(t=\tau,\xi)= \phi_\tau(\xi).
\end{array}\right. 
\end{equation}
For any $M>0$ with $\|\phi_\tau\|_{H^1(\mathbb{R})}\leq M$,
define the solution space on the time interval $[\tau, \tau+t_0]$ for $t_0\geq 0$ as
$$Y_{\tau, t_0}^M:= \bigg\{\phi\in C([\tau, \tau+t_0];H^1(\RR))\cap L^2(\tau,\tau+t_0; H^2(\RR))\bigg |\sup_{\tau\leq t\leq \tau+t_0} \|\phi\|_{H^1(\RR)}\leq M\bigg\}.$$ Now we define a mapping $\mathcal{T}$ on the solution space $Y_{\tau, t_0}^{2M}$. For any 
$\phi(t,\xi)\in Y_{\tau, t_0}^{2M}$, define $\tilde \phi(t,\xi):=\mathcal{T} \phi(t,\xi)$ be the solution to the linear equation
 \begin{equation}\label{-1}
\left\{\begin{array}{l}
\displaystyle \tilde \phi_t-s \tilde \phi_{\x}- \tilde\phi_{\x\x}=-\frac{4(\phi_\xi+U_\x)}{(u_--u_+)^2}\int_{\mathbb{R}}a(U) U_\xi  \phi d\xi-\big[f(\phi+ U)-f(U)\big]_\x, \quad t>\tau,\\[3mm]
\tilde  \phi(t=\tau,\xi)= \phi_\tau(\xi).
\end{array}\right. 
\end{equation}
First we prove that $\tilde \phi(t,\xi):=\mathcal{T} \phi(t,\xi)\in Y_{\tau, t_0}^{2M}$ for suitably small $t_0=t_0(M)>0$.
Multiplying the equation $\eqref{-1}_1$ by  $\tilde \phi(t,\xi)$ and integrating the resulting equation with respect to $t$ and $\x$ over $[\tau,\tau+t_0]\times \RR$, we can arrive at
  \begin{equation}\label{-3}
 \begin{aligned}
\frac{1}{2}\|\tilde  \phi\|^2(\tau+t_0)+\int_\tau^{\tau+t_0} \|\tilde  \phi_\x\|^2 dt=&\frac{1}{2}\|\phi_\tau\|^2-\int_\tau^{\tau+t_0}\int_\RR \tilde  \phi\big[f(\phi+ U)-f(U)\big]_\x d\x dt\\[3mm]
&  -4\int_\tau^{\tau+t_0}\int_\RR\frac{\tilde  \phi(\phi_\xi+U_\x)}{(u_--u_+)^2}d \x\int_{\mathbb{R}} a(U) U_\xi   \phi d\xi dt.
   \end{aligned}
    \end{equation}
Now we estimate the last two terms on the right hand side of \eqref{-3}. First it holds that
  \begin{equation}\label{-4}
 \begin{aligned}
&\left|-\int_\tau^{\tau+t_0}\int_\RR \tilde  \phi\big[f(\phi+ U)-f(U)\big]_\x d\x dt\right|\\
&=\left|\int_\tau^{\tau+t_0}\int_\RR \tilde  \phi\big[
f^\prime(\phi+U) \phi_\x-\big(f^{\prime}(\phi+U)-f^\prime(U)\big) U_\xi \big]d\x dt\right|\\
&=\left|\int_\tau^{\tau+t_0}\int_\RR \tilde  \phi\big[
f^\prime(\phi+U) \phi_\x-f^{\prime\prime}(\theta_1) U_\xi \phi\big]d\x dt\right|\\
&\leq C(M)\int_\tau^{\tau+t_0}\|\tilde  \phi\|\big[\| \phi\|
+\| \phi_\x\|\big]dt\\
&\leq \frac{1}{16}\displaystyle \sup_{\tau\leq t\leq \tau+ t_0}\|\tilde  \phi\|^2+C(M)t_0^2\|\phi\|^2_{H^1(\RR)},
 \end{aligned}
    \end{equation}
where and in the sequel $C(M)$ is a generic positive constant depending on $M$ and  we have used the fact 
$$\sup_{\tau\leq t\leq \tau+t_0}\|\phi\|_{L^\infty(\RR)}\leq C\sup_{\tau\leq t\leq \tau+t_0} \|\phi\|_{H^1(\RR)}\leq CM.$$
On the other hand, we have
      \begin{equation}\label{-5}
 \begin{aligned}
&\left|-4\int_\tau^{\tau+t_0}\int_\RR\frac{\tilde  \phi(\phi_\xi+U_\x)}{(u_--u_+)^2}d \x\int_{\mathbb{R}}  a(U) U_\xi \phi d\xi dt\right|\\
&\leq 4\left|\int_\tau^{\tau+t_0}\int_\RR\frac{\tilde  \phi\phi_\xi}{(u_--u_+)^2}d \x\int_{\mathbb{R}} a(U) U_\xi \phi d\xi dt\right|\\
&\ \ \ \ +4\left|\int_\tau^{\tau+t_0}\int_\RR\frac{\tilde  \phi U_\x}{(u_--u_+)^2}d \x\int_{\mathbb{R}} a(U) U_\xi \phi d\xi dt\right|\\
&\leq C\left|\int_\tau^{\tau+t_0}\|\tilde  \phi\|\| \phi_\xi\|\| \phi\|_{L^\infty(\RR)}dt\right|\\
&\ \ \ \ +C\left|\int_\tau^{\tau+t_0}\bigg(\int_\RR\frac{\tilde  \phi U_\x}{(u_--u_+)^2}d\x\bigg)^2+\bigg(\int_\RR  aU_\x \phi d\x\bigg)^2dt\right|\\
&\leq  \frac{1}{16}\displaystyle \sup_{\tau\leq t\leq \tau+ t_0}\|\tilde  \phi\|^2+C(M)t_0^2\|\phi_\xi\|^2+C\int_\tau^{\tau+t_0}\big(\|\tilde  \phi\|^2+\| \phi\|^2\big) dt\\
&\leq  \frac{1}{16}\displaystyle \sup_{\tau\leq t\leq \tau+t_0}\|\tilde  \phi\|^2+C(M)(t_0+t_0^2)\|\phi\|^2_{H^1(\RR)}+Ct_0\displaystyle \sup_{\tau\leq t\leq \tau+t_0}\|\tilde \phi\|^2.
 \end{aligned}
    \end{equation}
Substituting \eqref{-4} and \eqref{-5} into \eqref{-3}, and then choosing a suitably small $t_0=t_0(M)>0$, we have
 \begin{equation}\label{-6}
 \begin{aligned}  
  \displaystyle \sup_{\tau\leq t\leq \tau+t_0}\|\tilde \phi\|^2+\int_\tau^{\tau+t_0} \|\tilde  \phi_\x\|^2 dt
\leq 2M^2.
  \end{aligned}
    \end{equation}  
    Multiplying the equation $\eqref{-1}_1$ by  $-\tilde \phi_{\xi\xi}$ and then integrating the resulting equation with respect to $t$ and $\x$ over $[\tau,\tau+t_0]\times \RR$, we can obtain
  \begin{equation}\label{--3}
 \begin{aligned}
\frac{1}{2}\|\tilde  \phi_\x\|^2(\tau+t_0)+\int_\tau^{\tau+t_0} \|\tilde  \phi_{\x\x}\|^2 dt=&\frac{1}{2}\|(\phi_{\tau})_\xi\|^2+\int_\tau^{\tau+t_0}\int_\RR \tilde  \phi_{\x\x}\big[f(\phi+ U)-f(U)\big]_\x d\x dt\\
&  +4\int_\tau^{\tau+t_0}\int_\RR\frac{\tilde  \phi_{\x\x}(\phi_\xi+U_\x)}{(u_--u_+)^2}d \x\int_{\mathbb{R}} a(U) U_\xi \phi d\xi dt.
   \end{aligned}
    \end{equation}
 First, it holds that
      \begin{equation}\label{--4}
 \begin{aligned}
  &  \left|\int_\tau^{\tau+t_0}\int_\RR \tilde  \phi_{\x\x}\big[f(\phi+ U)-f(U)\big]_\x d\x dt\right|\\
    &\leq \frac{1}{16}\int_\tau^{\tau+t_0}\|\tilde  \phi_{\x\x}\|^2dt  +C\int_\tau^{\tau+t_0}\int_\RR \left|\big[f(\phi+ U)-f(U)\big]_\x \right|^2d\x  dt \\
&  \leq\frac{1}{16}\int_\tau^{\tau+t_0}\|\tilde  \phi_{\x\x}\|^2dt+  C(M)\int_\tau^{\tau+t_0}\big(\|\phi\|^2+\|\phi_\x\|^2\big) dt\\
&  \leq \frac{1}{16}\int_\tau^{\tau+t_0}\|\tilde  \phi_{\x\x}\|^2dt +C(M)t_0 \sup_{\tau\leq t\leq \tau+t_0}\|\phi\|^2_{H^1(\RR)}.
       \end{aligned}
    \end{equation}
Next, we can calculate
    \begin{equation}\label{--5}
 \begin{aligned}
  & \left|4\int_\tau^{\tau+t_0}\int_\RR\frac{\tilde  \phi_{\x\x}\phi_\xi}{(u_--u_+)^2}d \x\int_{\mathbb{R}}a(U) U_\xi  \phi d\x dt\right|\\
  &\leq C\int_\tau^{\tau+t_0}\|\phi_\x\|\|\tilde   \phi_{\x\x}\|\|\phi\|_{L^\infty(\RR)}dt\\
&  \leq \frac{1}{16}\int_\tau^{\tau+t_0}\|\tilde  \phi_{\x\x}\|^2dt+C(M)\int_\tau^{\tau+t_0}\|\phi_\x\|^2 dt\\
&  \leq \frac{1}{16}\int_\tau^{\tau+t_0}\|\tilde  \phi_{\x\x}\|^2dt +C(M)t_0 \sup_{\tau\leq t\leq \tau+t_0}\|\phi\|^2_{H^1(\RR)},
       \end{aligned}
    \end{equation}
    and 
    \begin{equation}\label{--6}
 \begin{aligned}
  &\left|4\int_\tau^{\tau+t_0}\int_\RR\frac{\tilde  \phi_{\x\x}U_\xi}{(u_--u_+)^2}d \x\int_{\mathbb{R}} a(U) U_\xi \phi  d\x dt\right|\\
  &\leq C\int_\tau^{\tau+t_0}\|\tilde  \phi_{\x\x}\|\left|\int_\RR  a(U) U_\xi  \phi \right|d\x dt\\
&  \leq \frac{1}{16}\int_\tau^{\tau+t_0}\|\tilde  \phi_{\x\x}\|^2dt+Ct_0^2\displaystyle \sup_{\tau\leq t\leq \tau+t_0}\|\phi\|_{L^\infty(\RR)}^2\\
&  \leq \frac{1}{16}\int_\tau^{\tau+t_0}\|\tilde  \phi_{\x\x}\|^2dt+Ct_0^2 \sup_{\tau\leq t\leq \tau+t_0}\|\phi\|^2_{H^1(\RR)}.
       \end{aligned}
    \end{equation}
Substituting \eqref{--4}, \eqref{--5} and \eqref{--6} into \eqref{--3}, and then choosing a suitably small $t_0=t_0(M)>0$, we have
 \begin{equation}\label{--8}
 \begin{aligned}  
  \displaystyle \sup_{\tau\leq t\leq \tau+t_0}\|\tilde \phi_\x\|^2+\int_\tau^{\tau+t_0} \|\tilde  \phi_{\x\x}\|^2 dt
  \leq 2M^2.
  \end{aligned}
    \end{equation}  
    By \eqref{-6} and \eqref{--8}, we have 
     \begin{equation}\label{--9}
 \begin{aligned}  
  \displaystyle \sup_{\tau\leq t\leq \tau+t_0}\|\tilde \phi\|^2_{H^1(\RR)}+\int_\tau^{\tau+t_0} \|\tilde  \phi_\x\|^2_{H^1(\RR)} dt\leq 4M^2.
  \end{aligned}
    \end{equation}  
Therefore, we proved $\tilde\phi\in Y_{\tau, t_0}^{2M}$. 
Next we prove that the above mapping $\mathcal{T}: Y_{\tau, t_0}^{2M}\to Y_{\tau, t_0}^{2M}$ is a contraction mapping in $C([\tau, \tau+t_0]; H^1(\RR))$ for suitably small $t_0>0$.
For this, $\forall \phi_1, \phi_2\in Y_{\tau, t_0}^{2M}$, denote $\Phi:= \phi_2- \phi_1$ and  $\tilde \Phi:=\mathcal{T}\phi_2-\mathcal{T}\phi_1=\tilde \phi_2-\tilde \phi_1$. Then we have 
  \begin{equation}\label{-8}
\left\{
\begin{array}{l}
\di \tilde \Phi_t-\tilde \Phi_{\x\x}-s\tilde \Phi_\x=-[f( \phi_2+U)-f( \phi_1+U)]_\x-\frac{4U_\x}{(u_--u_+)^2}\int_\RR a(U)U_\x \Phi d\x\\[3mm]
\di \qquad \qquad \qquad \qquad- \frac{4\phi_{2\x}}{(u_--u_+)^2}\int_\RR a(U)U_\x \Phi d\x- \frac{4\Phi_\x}{(u_--u_+)^2}\int_\RR   a(U)U_\x \phi_1 d\x,\\[3mm]
\di \tilde  \Phi(t=\tau,\xi)= 0.\\
  \end{array}
\right.
    \end{equation} 
Multiplying the equation $\eqref{-8}_1$ by  $\tilde \Phi$, and then integrating the resulting equation with respect to $t$ and $\x$ over $[\tau,\tau+t_0]\times \RR$, we have 
       \begin{equation}\label{-10}
 \begin{aligned}  
\frac{1}{2}\|\tilde \Phi\|^2(\tau+t_0)+\int_\tau^{\tau+t_0} \|\tilde \Phi_\x\|^2dt=&-\int_\tau^{\tau+t_0}\int_\RR\tilde \Phi[f( \phi_2+U)-f( \phi_1+U)]_\x d\x dt\\
&-4\int_\tau^{\tau+t_0}\int_\RR\frac{\tilde \Phi U_\x}{(u_--u_+)^2}d\x\int_\RR a(U)U_\x \Phi d\x dt\\
&-4 \int_\tau^{\tau+t_0}\int_\RR\frac{\phi_{2\x}\tilde \Phi}{(u_--u_+)^2}d\x\int_\RR a(U)U_\x \Phi d\x dt\\
&- 4\int_\tau^{\tau+t_0}\int_\RR\frac{\Phi_\x\tilde \Phi}{(u_--u_+)^2}d\x\int_\RR   a(U)U_\x \phi _1 d\x dt.
  \end{aligned}
    \end{equation} 
We estimate the right hand side of \eqref{-10} terms by terms. First, we have
        \begin{equation}\label{-11}
 \begin{aligned}  
& \left|-\int_\tau^{\tau+t_0}\int_\RR\tilde  \Phi[f( \phi_2+U)-f( \phi_1+U)]_\x d\x dt\right|\\
&= \left|\int_\tau^{\tau+t_0}\int_\RR\tilde  \Phi_\x[f( \phi_2+U)-f( \phi_1+U)] d\x dt\right|\\
&\leq C(M)\int_\tau^{\tau+t_0}\|\tilde \Phi_\x\|\|\Phi\|dt\\
&\leq \frac{1}{8}\int_\tau^{\tau+t_0}\|\tilde \Phi_\x\|^2d\x+ C(M) t_0\displaystyle \sup_{\tau\leq t\leq \tau+t_0}\| \Phi\|^2.
  \end{aligned}
    \end{equation} 
Then we have 
    \begin{equation}\label{-12}
 \begin{aligned}  
&\left|-4\int_\tau^{\tau+t_0}\int_\RR\frac{ U_\x \tilde \Phi }{(u_--u_+)^2}d\x\int_\RR  a(U)U_\x \Phi d\x  dt\right|\\
&\leq C\left|\int_\tau^{\tau+t_0} {\bigg(\int_\RR{U_\x \tilde \Phi  }d\x\bigg)^2}+{\bigg(\int_\RR  a(U)U_\x \Phi d\x\bigg)^2} dt\right| \\
&\leq Ct_0\big(\sup_{\tau\leq t\leq \tau+t_0}\|\tilde \Phi\|^2+\sup_{\tau\leq t\leq \tau+t_0}\| \Phi\|^2\big).
  \end{aligned}
    \end{equation} 
Finally, it holds that  \begin{equation}\label{-13}
 \begin{aligned}  
&\left|-4\int_\tau^{\tau+t_0}\int_\RR\frac{\phi_{2\x}\tilde \Phi}{(u_--u_+)^2}d\x\int_\RR a(U)U_\x \Phi d\x dt\right|\\
&\leq C\left|\int_\tau^{\tau+t_0} {\left(\int_\RR{\phi_{2\x}\tilde \Phi }d\x\right)^2}+{\left(\int_\RR  a(U)U_\x \Phi  d\x\right)^2} dt\right|\\
&\leq C(M)t_0\big(\sup_{\tau\leq t\leq \tau+t_0}\|\tilde \Phi\|^2+\sup_{\tau\leq t\leq \tau+t_0}\| \Phi\|^2\big),
\end{aligned}
    \end{equation}
and
\begin{equation}\label{-14}
 \begin{aligned}  
& \left|-4 \int_\tau^{\tau+t_0}\int_\RR\frac{\Phi_\x\tilde \Phi}{(u_--u_+)^2}d\x\int_\RR   a(U)U_\x \phi_1 d\x dt\right|\\
& \leq C(M)\int_\tau^{\tau+t_0}\|\Phi_\x\|\|\tilde \Phi\|dt\\
 &\leq \frac{1}{8}\sup_{\tau\leq t\leq \tau+t_0}\|\tilde \Phi\|^2+ C(M)t_0^2\displaystyle \sup_{\tau\leq t\leq \tau+t_0}\|\Phi_\x\|^2.
\end{aligned}
    \end{equation}
Substituting \eqref{-11}-\eqref{-14} into \eqref{-10}  and then taking $t_0$ smaller than before if needed,  we have 
\begin{equation}\label{-15}
\displaystyle \sup_{\tau\leq t\leq \tau+t_0}\|\tilde \Phi\|^2\leq \frac{1}{3}\displaystyle \sup_{\tau\leq t\leq \tau+t_0}\|\Phi\|^2_{H^1(\RR)}.
   \end{equation}
 Multiplying the equation $\eqref{-8}_1$ by  $-\tilde \Phi_{\x\x}$, and then integrating the resulting equation with respect to $t$ and $\x$ over $[\tau,\tau+t_0]\times \RR$, we can get
        \begin{equation}\label{--10}
 \begin{aligned}  
\frac{1}{2}\|\tilde \Phi_\x\|^2(\tau+t_0)+\int_\tau^{\tau+t_0} \|\tilde \Phi_{\x\x}\|^2dt=&\int_\tau^{\tau+t_0}\int_\RR\tilde \Phi_{\x\x}[f( \phi_2+U)-f( \phi_1+U)]_\x d\x dt\\
&+4\int_\tau^{\tau+t_0}\int_\RR\frac{\tilde \Phi_{\x\x} U_\x}{(u_--u_+)^2}d\x\int_\RR a(U)U_\x \Phi d\x dt\\
&+4 \int_\tau^{\tau+t_0}\int_\RR\frac{\phi_{2,\x}\tilde \Phi_{\x\x}}{(u_--u_+)^2}d\x\int_\RR  a(U)U_\x \Phi d\x dt\\
&+4\int_\tau^{\tau+t_0}\int_\RR\frac{\Phi_\x\tilde \Phi_{\x\x}}{(u_--u_+)^2}d\x\int_\RR   a(U)U_\x \phi_1 d\x dt.
  \end{aligned}
    \end{equation}   
We need to estimate the right hand side of \eqref{--10}. First, it holds that
     \begin{equation}\label{--11}
 \begin{aligned}  
& \left|\int_\tau^{\tau+t_0}\int_\RR\tilde  \Phi_{\x\x}[f( \phi_2+U)-f( \phi_1+U)]_\x d\x dt\right|\\
&\leq \frac{1}{16}\int_\tau^{\tau+t_0}\|\tilde \Phi_{\x\x}\|^2dt+ C(M)\int_\tau^{\tau+t_0}\big(\|\Phi\|^2+\|\Phi_\x\|^2\big) dt\\
&\leq \frac{1}{16}\int_\tau^{\tau+t_0}\|\tilde \Phi_{\x\x}\|^2dt+C(M)t_0\displaystyle\sup_{\tau\leq t\leq \tau+t_0}(\| \Phi\|^2+\|\Phi_\x\|^2).
 \end{aligned}  
    \end{equation} 
Then we can compute that
    \begin{equation}\label{--12}
 \begin{aligned}  
&\left|4\int_\tau^{\tau+t_0}\int_\RR\frac{\tilde \Phi_{\x\x} U_\x}{(u_--u_+)^2}d\x\int_\RR  a(U)U_\x \Phi  d\x dt\right|\\
&\leq \displaystyle \frac{1}{16}\sup_{\tau\leq t\leq \tau+t_0}\| \tilde \Phi_{\xi\xi}\|^2+Ct_0^2\|\Phi\|^2_{L^\infty(\RR)}\\
&\leq  \displaystyle \frac{1}{16}\sup_{\tau\leq t\leq \tau+t_0}\| \tilde \Phi_{\xi\xi}\|^2+Ct_0^2\sup_{\tau\leq t\leq \tau+t_0}\|\Phi\|^2_{H^1(\RR)},
  \end{aligned}
    \end{equation} 
    \begin{equation}\label{--13}
 \begin{aligned}  
&\left|4\int_\tau^{\tau+t_0}\int_\RR\frac{\phi_{2\x}\tilde \Phi_{\x\x}}{(u_--u_+)^2}d\x\int_\RR a(U)U_\x \Phi d\x dt\right|\\
&\leq C\int_\tau^{\tau+t_0}\|\tilde \Phi_{\x\x}\|\|\phi_{2\x}\|\|\Phi\|_{L^\infty(\RR)}dt \\
&\leq  \frac{1}{16}\int_\tau^{\tau+t_0}\|\tilde \Phi_{\x\x}\|^2dt+C(M)t_0\sup_{\tau\leq t\leq \tau+t_0}\|\Phi\|^2_{H^1(\RR)},
\end{aligned}
    \end{equation}
and
\begin{equation}\label{--14}
 \begin{aligned}  
& \left| 4\int_\tau^{\tau+t_0}\int_\RR\frac{\Phi_\x\tilde \Phi_{\x\x}}{(u_--u_+)^2}d\x\int_\RR   a(U)U_\x  \phi_1d\x dt\right|\\
& \leq C\int_\tau^{\tau+t_0}\|\Phi_\x\|\|\tilde \Phi_{\x\x}\|\|\phi_1\|_{L^\infty(\RR)}dt\\
 &\leq  \displaystyle \frac{1}{16}\int _\tau^{\tau+t_0}\| \tilde \Phi_{\xi\xi}\|^2d t+C(M)t_0\sup_{\tau\leq t\leq \tau+t_0}\|\Phi\|^2_{H^1(\RR)}.
\end{aligned}
    \end{equation}
    Substituting  \eqref{--11}-\eqref{--14} into \eqref{--10}  and then taking $t_0$ suitably smaller than before if needed,  we have 
\begin{equation}\label{--15}
\displaystyle \sup_{\tau\leq t\leq \tau+t_0}\|\tilde \Phi_\x\|^2\leq \frac{1}{3}\displaystyle \sup_{\tau\leq t\leq \tau+t_0}\|\Phi\|^2_{H^1(\RR)}.
   \end{equation}
By \eqref{-15} and \eqref{--15}, we have 
\begin{equation}\label{--16}
\displaystyle \sup_{\tau\leq t\leq \tau+t_0}\|\tilde \Phi\|^2_{H^1(\RR)}\leq \frac{2}{3}\displaystyle \sup_{\tau\leq t\leq \tau+t_0}\|\Phi\|^2_{H^1(\RR)}.
   \end{equation}
Hence, the mapping  $\mathcal{T}: Y_{\tau, t_0}^{2M}\to Y_{\tau, t_0}^{2M}$ is a contraction mapping in $C([\tau, \tau+t_0]; H^1(\RR))$ for suitably small $t_0=t_0(M)>0$.
Therefore, there exist a suitably small $t_0=t_0(M)>0$ and a unique fixed point $\phi\in Y_{\tau, t_0}^{2M}\subset C([\tau, \tau+t_0]; H^1(\RR))$ for the mapping $\mathcal{T}$, that is, $$ \phi(t,\xi)= \mathcal{T}\phi(t,\xi).$$
Then we  finish the proof of Proposition  \ref{2.1}.

\vskip0.3cm
\noindent\textbf{Acknowledgment:}   A. Vasseur was partially supported by the NSF grants DMS 2219434 and 2306852. Y. Wang supported by NSFC (Grant No. 12171459, 12288201, 12090014, 12421001) and CAS Project for Young Scientists in Basic Research, Grant No. YSBR-031.

\vskip0.3cm
\noindent{\bf Conflict of Interest:} The authors declared that they have no conflicts of interest to this work.

\vskip0.3cm
\noindent{\bf Availability of data and material:} Data sharing not applicable to this article as no datasets were generated or analyzed during the current study.

   \end{document}